\documentclass[11pt]{amsart}
\usepackage{amsthm, amssymb, latexsym, amsmath, mathtools, color}
\usepackage[a4paper, total={6.5in, 8in}]{geometry}
\usepackage[dvipsnames,svgnames,table]{xcolor}
\usepackage[colorlinks=true,linkcolor=RoyalBlue,urlcolor=RoyalBlue,citecolor=PineGreen]{hyperref}
\usepackage{tikz}
\usepackage{tikz-cd} 
\usetikzlibrary{arrows.meta}
\usepackage{comment, caption}
\usepackage{graphicx,subcaption}

\usepackage{pgf,tikz,pgfplots}
\pgfplotsset{compat=1.14}
\usepackage{pgf,tikz,pgfplots}
\usepackage{mathrsfs}
\usetikzlibrary{arrows}
\usepackage{mathrsfs}
\usetikzlibrary{arrows}
\usepackage{capt-of}
\usetikzlibrary{decorations.pathreplacing}
\usetikzlibrary{cd}
\usetikzlibrary{positioning}
\usetikzlibrary{calc}
\usepackage{algorithmic}
\usepgfplotslibrary{fillbetween}
\usetikzlibrary{intersections}
\usetikzlibrary{patterns}

\usepackage[nameinlink]{cleveref}

\title{A general-position problem for planar line arrangements}

\author{Oliver Roche-Newton}
\address{Oliver Roche-Newton \\ Institute for Algebra, Johannes Kepler University Linz, Linz, Austria.}
\email{o.rochenewton@gmail.com}

\newcommand{\cA}{\mathcal{A}}

\newcommand{\cK}{\mathcal{K}}
\newcommand{\cL}{\mathcal{L}}
\newcommand{\cH}{\mathcal{H}}
\newcommand{\cF}{\mathcal{F}}
\newcommand{\cG}{\mathcal{G}}
\newcommand{\cJ}{\mathcal{J}}
\newcommand{\cE}{\mathcal{E}}

\newcommand{\cC}{\mathcal{C}}

\newtheorem{lemma}{Lemma}

\newtheorem{theorem}{Theorem}
\newtheorem{corollary}{Corollary}

\theoremstyle{remark}

\begin{document}

\begin{abstract}
For all $\delta>0$ and infinitely many $n \in \mathbb N$, we show that there exists a set $L$ of $n$ lines in $\mathbb R^2$ such that there are no intersecting quadruples, but for every subset $L' \subset L$ such that $|L'| \geq n^{\frac{4}{5}+\delta}$, there exist three lines from $L'$ with a common point of intersection. This gives an improved bound for a dual form of a theorem of Balogh and Solymosi \cite{BaSo}. As a consequence, we derive an improved lower bound for the Hadwiger-Debrunner number $HD_2(p,3)$.

We also give, for all $0 \leq s \leq 1$ and arbitrarily large $n \in \mathbb N$, a construction of a point set $S \subset [n]^3$ with cardinality $|S|\geq n^{3-s}$, such that $S$ contains $O(n^{6-4s})$ collinear triples. This shows that a supersaturation lemma of Balogh and Solymosi \cite{BaSo} is optimal, up to logarithmic factors.
\end{abstract}

\maketitle

\section{Introduction}

Let $P \subset \mathbb R^2$ be a finite point set with no four points on a line. Is it always true that $P$ contains a large subset with no three points on a line?  In  \cite{BaSo}, Balogh and Solymosi considered this problem and gave a partial negative answer. 

\begin{theorem} \label{thm:BS}
Let $\delta >0$. For infinitely many $n \in \mathbb N$, there exists a set $P \subset \mathbb R^2$ such that $|P|=n$, no four points of $P$ lie on a line, and every subset $P'$ with $|P'| \geq n^{\frac{5}{6}+ \delta}$ contains a collinear triple.
\end{theorem} 

On the other hand, Füredi \cite{Fu} proved that such a set $P$ must contain a collinear-triple-free subset $P' \subset P$ such that $|P'| =\Omega(\sqrt {n \log n})$, slightly improving the elementary bound $|P'| =\Omega( \sqrt {|P|})$ which can be obtained via a greedy algorithm, or from a simple probabilistic deletion argument. The question of determining the correct exponent for this problem\footnote{See also T. F. Bloom, Erdős Problem \#589, https://www.erdosproblems.com/589, accessed 2026-07-13.} remains wide open.

We consider a dual form of this problem in this paper. Given a family of sets $\cF$ and a positive integer $k$, we say that $\mathcal F$ is $k$-independent if for every subset $\{F_1,\dots,F_k\} \subseteq \cF$ with cardinality $k$ we have $F_1\cap  \dots \cap F_k= \emptyset$.

After applying point-line duality, Theorem \ref{thm:BS} shows that there exists a family $L$ of $n$ lines in $\mathbb R^2$ such that $L$ is $4$-independent, but every subset $L' \subseteq L $ with $|L'| \geq n^{\frac{5}{6}+\delta} $ is not $3$-independent. The main result of this paper is the following, which improves the exponent $5/6$ to $4/5$ for this form of the problem.

\begin{theorem} \label{thm:BSlines}
    For all $\delta>0$ and for infinitely many $n \in \mathbb N$, there exists a set $L$ of $n$ lines in $\mathbb R^2$ such that $L$ is $4$-independent, but for every subset $L' \subseteq L$ such that $|L'| \geq n^{\frac{4}{5}+\delta}$, there exist three elements of $L'$ with a common point of intersection.
\end{theorem}

Curiously and somewhat frustratingly, despite the apparent duality of these two problems, Theorem \ref{thm:BSlines} does not imply a direct improvement to Theorem \ref{thm:BS} itself. The reason for this is that the family of lines $L$ given by Theorem \ref{thm:BSlines} contains large subfamilies of parallel lines. After dualising the problem back to its original form, these bunches of parallel lines become lines with many points on them, and so the property that the line set has no intersecting quadruples is not preserved as a no-four-on-a-line property under dualisation. This highlights a subtle distinction between this problem in the projective and affine settings.

A more general form of this question is the following: given a $4$-independent family $\cF$ of convex sets in $\mathbb R^2$, does there always exist a large subset $\cF' \subset \cF$ such that $\cF'$ is $3$-independent?  Theorem \ref{thm:BSlines} gives a partial negative answer to this problem. On the other hand, a result of Kalai \cite{Ka} implies that any $4$-independent family of convex sets with cardinality $n$ determines at most $\binom{n-1}{2}$ intersecting triples, and a standard probabilistic deletion argument then implies that there exists a $3$-independent subset with cardinality $\Omega(n^{1/2})$. So, the correct exponent for this problem is somewhere in the range $[\frac{1}{2},\frac{4}{5}+\delta]$. This question was considered by Keller and Smorodinsky \cite{KeSm}, where the dual form of the Balogh-Solymosi Theorem was used to give a lower bound for the Hadwiger-Debrunner number $HD_2(p,3)$. Theorem \ref{thm:BSlines} implies a further improvement to their lower bound.

\begin{corollary} \label{cor:HD2}
    For all $\epsilon >0$ and all $p$ sufficiently large, $HD_2(p,3) = \Omega( p^{\frac{5}{4}- \epsilon})$.
\end{corollary}
Full details of this, including the definitions necessary to understand Corollary \ref{cor:HD2} and its proof, are given in Section \ref{sec:HD}.

A key step in the work of Balogh and Solymosi was the following supersaturation lemma for collinear triples in the three dimensional grid $[n]^3$. Given $S \subset \mathbb R^3$, let $t(S)$ denote the number of collinear triples in $S$, i.e. the number of three-element subsets of $S$ such that all elements of the set are collinear.

\begin{theorem}[Lemma 4.3 in \cite{BaSo}] \label{thm:ssl}
    Let $0 \leq s \leq 1$ and let $S \subset [n]^3$ be a set of size $|S|\geq n^{3-s}$. Then $t(S) \gg \frac{n^{6-4s}}{\log n}$.
\end{theorem}

This lemma is an important contributing factor to establishing the main result of \cite{BaSo}, as the supersaturation lemma is needed for the application of the hypergraph container method to work. It was not clear if Theorem \ref{thm:ssl} was optimal; to the best of our knowledge, a $p$-random subset with $p=n^{-s}$ yielding around $n^{6-3s}$ collinear triples was the example giving the fewest collinear triples. Since any improvement to Theorem \ref{thm:ssl} would improve the exponent in Theorem \ref{thm:BS}, this motivated attempts to get a better exponent than $6-4s$. It turns out that this is not possible; we give a construction showing that Theorem \ref{thm:ssl} is in fact optimal up to logarithmic factors.

    

\begin{theorem} \label{thm:main}

     Let $p\equiv 3 \pmod 4$ be prime, and let $r\ge 1$ be an integer. Then there exists a set $S\subseteq [pr]^3$ with $ |S|=p^2r^3$ such that
\[
t(S)\ll p^2r^6.
\]

\end{theorem}

To compare this more cleanly with Theorem \ref{thm:ssl}, we present the following corollary.

\begin{corollary} \label{cor:main}
    For every fixed $0\le s\le 1$, there are arbitrarily large $n$ for which there exists a set $S\subseteq[n]^3$ with
\[
|S|\geq n^{3-s}
\]
and
\[
t(S)\ll n^{6-4s}.
\]
\end{corollary}

 \subsubsection*{Notation}
 
Throughout the paper, the standard notation $\ll, \,\gg$ and respectively $O$ and $\Omega$ is applied to positive quantities in the usual way. That is, $X\gg Y$, $Y \ll X,$ $X=\Omega(Y)$ and $Y=O(X)$ all mean that $X\geq cY$, for some absolute constant $c>0$. If both $X=O(Y)$ and $X= \Omega(Y)$ hold then we write $X= \Theta(Y)$. 
All logarithms are in base $2$, unless stated otherwise. The notation $[n]$ refers to the set of the first $n$ positive integers. Given a set $X$, the notation $2^X$ is used for the power set of $X$. Given a hypergraph $\cH=(V,E)$ and any $V' \subset V$, $\cH[V']$ denotes the subhypergraph induced by $V'$, i.e. the hypergraph with vertex set $V'$ whose edge set consists of all the edges of $\cH$ contained in $V'$.

\section{Supersaturation}

The starting point for the proof of Theorem \ref{thm:BSlines} is to define an ambient set of lines $\mathcal L$, within which our final set of lines $L$ will belong. In fact, choosing a good ambient set is the key new contribution to the proof of Theorem \ref{thm:BSlines}, as much of the rest of the proof then follows the framework from \cite{BaSo} closely. In this context, a ``good" family $\cL$ should satisfy the following properties.
\begin{itemize}
    \item We need a quantitatively strong supersaturation result, showing that large subsets of $\cL$ must determine many intersecting triples. This is the content of the forthcoming Lemma \ref{lem:sslL}.
    \item We need a set of lines $\cL$ that is rich in intersecting triples, but not particularly rich in intersecting quadruples. This is the content of the forthcoming Lemma \ref{lem:counting}.
\end{itemize}
For comparison, the proof of Balogh and Solymosi starts out with an ambient set of points $[n]^3$. This choice does particularly well for the analogue of the second point above; the number of collinear quadruples determined by this set is of almost the same order as the number of collinear triples. However, the supersaturation for collinear triples in this set is the main limitation  as to why the exponent $5/6+\delta$ in Theorem \ref{thm:BS} cannot be improved further. It turns out that the supersaturation lemma from \cite{BaSo} is essentially optimal, as we will show in Section \ref{sec:optimal}.

For an arbitrary set $L$ of lines in $\mathbb R^2$, we use $T(L)$ to denote the set of all intersecting triples of lines from $L$, that is
\[
T(L):= \{ A \subset L : |A|=3, \text{ and the three lines in $A$ have a common intersection} \}.
\]
Similarly, let $Q(L)$ denote the set of all intersecting quadruples of lines from $L$.

For $a,b \in \mathbb R$, define $\ell_{a,b}$ to be the line with equation $y=ax+b$. Define
\begin{equation} \label{Ldefn}
\mathcal L:= \{ \ell_{a,b} : a \in [r], b \in [r^7] \},
\end{equation}
where $r$ is a parameter going to infinity. The main result of this section is the following supersaturation lemma for intersecting triples in subsets of $\cL$.

\begin{lemma} \label{lem:sslL}
Let $\mathcal J \subseteq \mathcal L$ such that $|\mathcal J| \geq 8r^7$. Then
\[
|T(\mathcal J)| \geq \frac{|\mathcal J|^3}{1000r^8}.
\]
\end{lemma}

\begin{proof}
    We restrict our attention to intersecting triples of $\mathcal J$ occurring in the set $P=[r^6] \times [2r^7]$. For $p \in P$, define
    \[
n(p):=| \{ \ell \in \mathcal J : p \in \ell \}|
    \]
    and note that
    \[
    \sum_{p \in P} n(p)= \sum_{\ell \in \mathcal J} |\ell \cap P|= |\mathcal J|r^6.
    \]
    The contribution to this sum from points with $n(p) \leq 2$ is at most $4r^{13}$. It then follows from the assumed lower bound on $|\mathcal J|$ that
    \[
    \sum_{p \in P: n(p) \geq 3} n(p) \geq  \frac{|\mathcal J|r^6}{2}.
    \]
    An application of Hölder's inequality gives
    \begin{align*}
    \frac{|\mathcal J|^3r^{18}}{8} \leq  \left (  \sum_{p \in P: n(p) \geq 3} n(p) \right )^3 \leq |P|^2 \left (\sum_{p \in P: n(p) \geq 3} n(p)^3 \right ) &\leq (2r^{13})^2 \cdot 27 \left (\sum_{p \in P: n(p) \geq 3} \binom{n(p)}{3} \right )
    \\& \leq 108 r^{26} |T(\mathcal J)|.
    \end{align*}
    A rearrangement of this inequality completes the proof.
\end{proof}

Note that the condition that $|\cJ| \geq 8r^7$ is necessary, and it is optimal up to the constant, since $\cL$ contains subsets of size $r^7$ such that all of the lines in the subset are parallel, and such a subset determines zero triple intersections.

We now record here some information about the number of intersecting triples and quadruples determined by $\cL$.

\begin{lemma} \label{lem:counting}
    For the set $\cL$ of lines defined in \eqref{Ldefn}
    \[
    |T( \cL)| \leq c_1 r^{16} \log r \,\, \text{and} \,\,\,  |Q(\cL)| \leq c_2 r^{17},
   \]
    where $c_1,c_2 \geq 1$ are absolute constants.
\end{lemma}

\begin{proof}
    Similarly to the previous proof, let $p=(p_1,p_2) \in \mathbb R^2$ and consider
    \[
n(p):=| \{ \ell \in \mathcal L : p \in \ell \}|.
    \]
    This quantity now counts \textit{all} lines in $\cL$ which pass through $p$. Observe that
    \[
    n(p)= |\{(a,b) \in [r] \times [r^7] : p_2=ap_1+b \}|.
    \]
    Therefore, $n(p)$ is equal to the number of points from $S=[r] \times  [r^7]$ on the line with equation 
    \begin{equation} \label{dualline}
    y=-p_1x+p_2.
    \end{equation}
    It follows that the number of intersecting triples occurring at $p$ is equal to the number of collinear triples from $S$ on this line. Repeating this for all $p \in \mathbb R^2$, it follows that $|T(\cL)|$ is equal to the number of collinear triples in $S$, excluding the collinear triples on vertical lines (since the line defined in \eqref{dualline} can never be vertical). It remains to bound this quantity.

   We first note that the number of collinear triples in $S$ on horizontal lines is at most $r^{10}$, and so it will be sufficient to upper bound the number of collinear triples on non-horizontal and non-vertical lines.

   Fix a point $q \in S$. We will count the number of collinear triples on the lines through $q$ and then sum over all $q \in S$. A non-vertical and non-horizontal line through $q$ which contains at least one other point of $S$ can be written as
   \[
   \ell= \{ q +  t(u,v) : t \in \mathbb R \}
   \]
   where $u \in \mathbb [r]$, $v \in \{-r^7,\dots,r^7\} \setminus \{0\}$ and $\gcd(u,v)=1$.  This line contains at most $\min \left \{ \frac{r}{u}, \frac{r^7}{|v|} \right \}$ other points from $S$. Summing over all possible directions (and discarding the requirement that $u$ and $v$ are coprime), it follows that the number of collinear triples involving $q$ is at most
   \[
   \sum_{u=1}^{r} \sum_{v \in \{-r^7,\dots,r^7\} \setminus \{0\}}  \left(\min \left \{ \frac{r}{u}, \frac{r^7}{|v|} \right\}  \right )^2.
   \]
   A dyadic decomposition then gives that the number of collinear triples involving $q$ is at most
   \begin{align*}
      \sum_{u=1}^{r} \sum_{v \in \{-r^7,\dots,r^7\} \setminus \{0\}} \left (\min \left \{ \frac{r}{u}, \frac{r^7}{|v|} \right \} \right )^2 & \leq  \sum_{j=0}^{\lceil\log r \rceil} \sum_{(u,v) \in \mathbb N \times  \mathbb Z^* : 2^j \leq \min \left \{ \frac{r}{u}, \frac{r^7}{|v|} \right \} < 2^{j+1}} (2^j)^2
      \\& \ll \sum_{j=0}^{ \lceil \log r \rceil} \frac{r}{2^j} \frac{r^7}{2^j}(2^j)^2 \ll r^{8} \log r.
   \end{align*}
   Summing over all $q \in S$, it follows that there are at most $O(r^{16}\log r)$ collinear triples in $S$ on non-vertical lines, and thus $|T(\cL)| \ll r^{16} \log r$. 

   The argument bounding $|Q(\cL)|$ is similar. We need to count the number of collinear quadruples in $S$ on non-vertical lines. Horizontal lines contribute at most $r^{11}$ such quadruples. Fixing $q \in S$ and a primitive non-axis-parallel direction $d=(u,v)$, the number of collinear quadruples involving $q$ on the line through $q$ with direction $d$ is at most $\left (\min \left \{ \frac{r}{u}, \frac{r^7}{|v|} \right \} \right )^3$, and so
   \begin{align*}
     |Q(\cL)| &  \leq r^{11}+ r^8 \sum_{u=1}^{r} \sum_{v \in \{-r^7,\dots,r^7\} \setminus \{0\}} \left (\min \left \{ \frac{r}{u}, \frac{r^7}{|v|} \right \} \right )^3 
      \\& \leq r^{11}+ r^8 \sum_{j=0}^{\lceil \log_r \rceil } \sum_{(u,v) \in \mathbb N \times  \mathbb Z^* : 2^j \leq \min \left \{ \frac{r}{u}, \frac{r^7}{|v|} \right \} < 2^{j+1}} (2^j)^3
      \\& \ll r^{11}+  r^8 \sum_{j=0}^{ \lceil \log r \rceil } \frac{r}{2^j} \frac{r^7}{2^j}(2^j)^3
       = r^{11}+  r^{16} \sum_{j=0}^{ \lceil \log r \rceil} 2^j \ll r^{17}.
   \end{align*}
\end{proof}

The choice of the highly unbalanced parameter grid $[r]\times[r^7]$ defining the line set $\cL$ results from an optimization. More generally consider the family of lines
\[
 \{ \ell_{a,b} : a \in [m], b \in [n] \}.
\] 
This set gives rise to approximately $m^2n^2$ intersecting triples (up to a logarithmic factor) and $m^3n^2$ intersecting quadruples. If we choose $m$ to be much smaller than $n$, the number of triples and quadruples become closer, which is advantageous for the forthcoming container application. On the other hand, this family of lines contains large parallel pencils of cardinality $n$. If $n$ gets too large this becomes problematic for the argument, as these large parallel pencils become hard to avoid.

\section{Application of the container theorem}

As was the case in \cite{BaSo}, the proof of Theorem \ref{thm:BSlines} makes use of hypergraph containers. The theory of hypergraph containers was developed independently by Balogh, Morris and Samotij \cite{BMS} and Saxton and Thomason \cite{SaTh}. Roughly speaking, the statement that we use says that if a hypergraph has a reasonably good edge distribution, we can obtain strong quantitative information about where the independent sets of the hypergraph may be found. Results of this type have led to several major breakthroughs in combinatorics in recent years; see \cite{BMS2} for a survey of this topic. 

Before stating the result, we need to establish some notation. For a $3-$uniform hypergraph $\cH=(V,E)$ and $v \in V$, $d(v)$ denotes the degree of $v$, i.e. the number of edges which contain $v$. Let $d(\cH)$ denote the average degree of $\cH$, so
\begin{equation} \label{ddefn}
d(\cH)= \frac{1}{|V|}\sum_{v \in V} d(v)=  \frac{3|E|}{|V|}.
\end{equation}
We also define the \emph{co-degree} for a subset $S \subseteq V$ of vertices as
$$d(S) = |\{ e \in E(\cH) : S \subseteq e \}|.$$
Using this definition we define the \emph{maximum co-degree} $\Delta_2(\mathcal H)$ as
$$\Delta_2(\cH) = \max_{\substack{S \subseteq V \\ |S| = 2}}d(S).$$

We will use the following formulation of the container theorem, which is a special case of Corollary 3.6 in \cite{SaTh}.

\begin{theorem} \label{thm:container}
Let $\cG=(V,E)$ be a $3-$uniform hypergraph on $n$ vertices, and let $\epsilon, \tau \in (0,1/2)$. Suppose that
\begin{equation} \label{cond1} 
\frac{\Delta_2(\cG)}{d(\cG) \cdot \tau}+ \frac{1}{2d(\cG) \cdot \tau^{2}} \leq \frac{\epsilon}{288}
\end{equation}
and
\begin{equation} \label{cond2}
\tau < \frac{1}{21600}.
\end{equation}
Then there exists a set $\cC$ of subsets of $V$ such that
\begin{enumerate}
    \item if $A \subseteq V$ is an independent set then there exists $C \in \cC$ such that $A \subseteq C$;
    \item $|E(\cG[C])| \leq \epsilon |E(\cG)|$ for all $C \in \cC$;
    \item $\log |\cC| \leq c n  \tau \log(\frac{1}{\epsilon})  \log(\frac{1}{\tau})$,  
\end{enumerate}
where $c$ is an absolute constant.
\end{theorem}

We now define a $3$-uniform hypergraph $\mathcal H$ encoding triple intersection determined by the line set $\mathcal L$ defined in \eqref{Ldefn}. The vertex set of $\mathcal H$ is $\mathcal L$, and three distinct elements $\ell_1, \ell_2, \ell_3 \in  \mathcal L$ form an edge in $\mathcal H$ if and only if $\ell_1 \cap \ell_2 \cap \ell_3 \neq \emptyset$. Theorem \ref{thm:container} will be used to prove the following container lemma for independent sets in induced subhypergraphs of $\mathcal H$.

\begin{lemma} \label{lem:cont1}
   Let $r$ be a sufficiently large integer and let $\cL$ be the line set defined in \eqref{Ldefn}. Let $\mathcal J \subseteq \mathcal L$ such that $|\cJ| \geq 8r^7$ and let $\cG=\mathcal H[ \mathcal J]$ be the subhypergraph of $\mathcal H$ induced by $\mathcal J$. For all $0<\delta <1$, there exists a family of sets $\mathcal C_{\mathcal J} \subseteq 2^{\mathcal J}$ such that the following three statements hold.
    \begin{itemize}
        \item If $A \subseteq \mathcal J$ is an independent set in $\cG$ then there exists $C \in \mathcal C_{ \mathcal J}$ such that $A \subseteq C$.
        \item For all $C \in \mathcal C_{ \mathcal J}$,
        \[
        |E(\cH[C])| \leq r^{-\delta} |E( \cH[\cJ])|
        \]
        \item For some absolute constant $c'$, we have 
        \[
        \log | \cC_{\cJ}| \leq c'\delta r^{4+\delta} \log^2 r .
        \]
    \end{itemize}
\end{lemma}

\begin{proof} 
    Note first that 
    \begin{equation} \label{Deltabound}
    \Delta_2(\cG) \leq r.
    \end{equation} 
    Indeed, fix any two fixed lines $\ell_1,\ell_2 \in \cJ$. If they are parallel then they do not intersect and so $d(\{ \ell_1,\ell_2\})=0$. If they are not parallel then they have a unique point of intersection $p$. Since the slopes of the lines in $\cL$ come from a set with cardinality $r$, there exist at most $r$ lines from $\cL$ (and thus also at most $r$ lines from $\cJ$) which pass through $p$.

    Next, observe that Lemma \ref{lem:sslL} implies that
    \begin{equation} \label{dbound}
    d(\cG) = \frac{3|T(\cJ)|}{|\cJ|} \geq \frac{3|\cJ|^2}{1000r^8}.
    \end{equation}
    Apply Theorem \ref{thm:container} for the hypergraph $\cG$ with
    \[
    \tau=K \frac{r^{4+ \delta}}{|\cJ|}, \,\, \epsilon= r^{-\delta},
    \]
    where $K$ is a sufficiently large absolute constant. It then follows that the main condition \eqref{cond1} from Theorem \ref{thm:container} is satisfied. Indeed, by \eqref{Deltabound}, \eqref{dbound} and the assumed lower bound on $|\cJ|$ in the statement of the lemma, 
    \begin{align*}
    \frac{\Delta_2(\cG)}{d(\cG) \tau} + \frac{1}{2d(\cG) \tau^2} & \leq \frac{1000r^{5-\delta}}{3K|\cJ|} + \frac{1000}{6K^2r^{2\delta}}
     \leq \frac{1000}{3Kr^{\delta}} \leq \frac{1}{288r^{\delta}},
    \end{align*}
    where the last inequality is valid for $K$ sufficiently large.
    Condition \eqref{cond2} follows from the assumptions that $|\cJ| \geq 8r^7$ and that $r$ is sufficiently large. For some absolute constant $c'$, we obtain a set $\mathcal C_{\cJ} \subseteq 2^{\cJ}$ such that
    \[
    \log | \cC_{\cJ}| \leq c'\delta r^{4+\delta} \log ^2r.
    \]
    We also have that, for all $C \in  \cC_{\cJ}$,
    \[
    |E(\cH[C])|=|E( \cG[C])| \leq \epsilon |E(\cG)|=r^{-\delta}|E(\cH[\cJ])|.
    \]

    \end{proof}

By iteratively applying Lemma \ref{lem:cont1}, we obtain a family of containers for the hypergraph $\cH$ with the properties we need.

\begin{lemma} \label{lem:cont2}
   Let $0 < \delta < 1$ and let $r$ be sufficiently large. Let $\cL$ be the set of lines defined in \eqref{Ldefn} and let $\cH$ be the $3$-uniform hypergraph encoding intersecting triples of lines from $\cL$. There exists a set $\cC \subseteq 2^{\cL}$ such that the following three statements hold.
   \begin{itemize}
        \item If $A \subseteq \mathcal L$ is an independent set in $\cH$ then there exists $C \in \mathcal C$ such that $A \subseteq C$.
        \item For all $C \in \mathcal C$,
        \[
        |C| \leq 8r^7.
        \]
        \item For some constant $c$, we have 
        \[
        \log | \cC| \leq c\delta r^{4+\delta} \log^2 r .
        \]
    \end{itemize}
\end{lemma}

\begin{proof}


    Iteratively apply Lemma \ref{lem:cont1}. We begin by applying the lemma with $\cJ=\cL$ to obtain a set of containers $\cC_{1}$ satisfying 
    \[
    \log | \cC_1| \leq c' \delta r^{4+\delta} \log^2 r
    \]
    such that any independent set $A$ in $\cH$ satisfies $A \subseteq C$ for some $C \in \cC_{1}$. For each $C \in \cC_{1}$, there are two possibilities.
    \begin{itemize}
    \item If $|C| > 8r^7$ we apply Lemma \ref{lem:cont1} with $\cJ=C$. We obtain a family $\cC_{C}$ of subsets of $C$ such that every independent set in $\cH[C]$ is contained inside a set in $\cC_{C}$.
        \item If $|C| \leq 8r^7$ then we put $C$ into our final set of containers $\cC$. For the purpose of describing this algorithm, it is convenient to define $\cC_{C}= \{ C \}$.
    \end{itemize}
We define 
\[
\cC_2:= \bigcup_{C \in \cC_1} \cC_C
\]
Note that $\mathcal C_2$ is a container set for $\cH$. Indeed, suppose that $X$ is an independent set in $\cH$. Then there is some $C \in \mathcal C_1$ such that $X \subseteq C$. Also, $X$ is an independent set in the hypergraph $\cH[C]$, which implies that $X \subseteq C'$ for some $C' \in \mathcal C_C \subseteq \mathcal C_2$.
    
    We then repeat this process for each set in $\cC_2$ to obtain a larger family of containers $\cC_3$ with smaller component sets. We can visualise this as a tree of containers, where the final set of containers $\cC$ consists of all of the leaf nodes. This process terminates when all of the leaf nodes correspond to sets $C$ with $|C| \leq 8r^7$. 
    
    We claim that this process terminates after at most $\lceil \frac{4}{\delta} \rceil$ steps. Indeed, let $C$ be a node in this tree at depth at least $\lceil \frac{4}{\delta}\rceil $. Since each application of Lemma \ref{lem:cont1} reduces the number of edges by a factor of $r^{-\delta}$, it follows that the number of edges determined by $C$ satisfies
    \begin{equation} \label{edges}
    |E(\cH[C])| \leq c_1r^{16} \log r(r^{-\delta})^{\lceil \frac{4}{\delta}\rceil }\leq c_1r^{12} \log r.
    \end{equation}
     Now suppose for a contradiction that $|C| > 8r^7$. Lemma \ref{lem:sslL} then implies that
    \[
   | E(\cH[C])| \geq \frac{|C|^3}{1000r^8}
    \]
    and combining this with \eqref{edges} and rearranging yields
    \[
    |C| \leq c_3 r^{20/3} \log^{1/3} r
    \]
    for an absolute constant $c_3$. But $|C| > 8r^7$ and this is a contradiction for $r$ sufficiently large.

    At the end of this process, we have a set of containers $\cC$ such that the bound $|C| \leq 8r^7$ holds for all $C \in \cC$. Finally, it follows from the upper bound for the size of the container families given in Lemma \ref{lem:cont1} that
    \[
    |\cC| \leq \left(2^{c'\delta r^{4+\delta} \log^2 r} \right)^{\lceil \frac{4}{\delta}\rceil } \leq  2^{cr^{4+\delta} \log^2 r},
    \]
as required.
\end{proof}

\section{A random subset of $\cL$}

In this section, we turn to the task of proving Theorem \ref{thm:BSlines}. First, we use Lemma \ref{lem:cont2} to show that, for a suitable choice of $p$, a $p$-random subset of $\mathcal L$ does not contain a large $3$-independent subset with high probability. Given a family of sets $\cA$, we denote
\[
\alpha_3(\cA):= \max \{ |\cA'| : \cA' \subseteq \cA \text{ and } \cA' \text{ is $3$-independent} \}.
\]

\begin{lemma} \label{lem:whp}
 Fix $0<\delta<1$ and let $\cL$ be the set of lines defined in \eqref{Ldefn}. Let $0 \leq p \leq  r^{-3}$ and let $\cK$ be a $p$-random subset of $\cL$, defined by setting 
 \[
 \mathbb P\left [ \ell \in \mathcal K\right ]=p, \,\,\, \forall \, \ell \in \cL.
 \]
 Then
 \[
 \lim_{r \rightarrow \infty} \mathbb P \left [ \alpha_3(\cK) \geq r^{4+2\delta}\right ] =0
 \]

 \begin{proof}
    Let $m=r^{4+2\delta}$ and suppose for simplicity that $m$ is an integer. We need to estimate the probability that our $p$-random set $\cK$ contains a $3$-independent subset of size $m$, and these sets correspond precisely to independent sets of size $m$ in the hypergraph $\cH$. Lemma \ref{lem:cont2} tells us that the only candidates for independent sets with size $m$ in this graph are subsets of $C \in  \cC$ with cardinality $m$, with good quantitative information about the sizes of the container family $\cC$ and its components. Each such independent set belongs to $\cK$ with probability $p^m$, and it therefore follows from the union bound that
    \begin{align*}
        \mathbb P[ \alpha_3(\cK) \geq m] \leq \sum_{C \in \cC} \binom{|C|}{m} p^m
         \leq 2^{cr^{4+\delta} \log^2 r} \binom{8r^7}{m} p^m
        & \leq 2^{cr^{4+\delta} \log^2 r} \left (\frac{8er^7p}{m} \right)^m
        \\& \leq 2^{cr^{4+\delta} \log^2 r} \left (\frac{8e}{r^{2\delta}} \right)^{r^{4+2\delta}}
        \\& \leq \left (\frac{16e}{r^{2\delta}} \right)^{r^{4+2\delta}}.
    \end{align*}
    The third inequality above uses the fact that $\binom{a}{b} \leq \left ( \frac{ea}{b} \right)^b$, the fourth uses the assumed upper bound on $p$ from the statement, and the fifth is valid as long as $r$ is sufficiently large (with respect to the parameter $\delta$). This completes the proof.
 \end{proof}
 
\end{lemma}


We are now ready to prove Theorem \ref{thm:BSlines}.

\begin{proof}[Proof of Theorem \ref{thm:BSlines}]

It suffices to consider $\delta < \frac{1}{5}$, since the theorem is trivial outside this range. Let $p=\frac{1}{2c_2^{1/3}r^3}$, where $c_2$ is the absolute constant from the statement of Lemma \ref{lem:counting}. Construct a $p$-random subset $\cK \subseteq \cL$, each $\ell \in \cL$ belonging to $\cK$ with probability $p$. We will show that, with positive probability, all of the following statements are simultaneously true:
 \begin{enumerate}
 \item $|\cK| \geq \frac{1}{2} \mathbb E[|\cK|]=\frac{r^5}{4c_2^{1/3}}$,
     \item $|Q(\cK)| \leq\frac{|\cK|}{2}$, \label{point2}
     \item $\alpha_3(\cK) \leq r^{4+2\delta}$.
 \end{enumerate}
 
 Suppose that, with positive probability, all of these statements hold, and in particular there is some set $\cK$ with all of these properties. Fix this set $\cK$. We prune $\cK$ to find a subset $L \subseteq \cK$ such that $|L| \gg r^5$ and $L$ is $4$-independent. That is, for every intersecting quadruple given by $\cK$, remove one line to destroy this quadruple. Because of point \eqref{point2} above, at the end of this process, we have deleted at most half of the lines and have no intersecting quadruples. Also, since $\cK$ does not contain any $3$-independent subset with size $r^{4 +2\delta}$, $L$ also does not, so $\alpha_3(L) \leq r^{4+2\delta}$. Therefore, the set $L$ has all of the properties claimed in the statement of Theorem \ref{thm:BSlines}, and the proof is complete.

  It remains to prove that the three statements above hold simultaneously with positive probability. Let $\mathcal E_1$ be the event that $|\cK| \leq \frac{1}{2}\mathbb E[|\cK|]$. Note that $Var[|\cK|] \leq \mathbb E[|\cK|]$, and so by Chebyshev's Inequality
 \[
 \mathbb P[\mathcal E_1] \leq \mathbb P\left [ | |\cK| - \mathbb E[|\cK|] | \geq\frac{1}{2} \mathbb  E[|\cK|] \right] \leq \frac{4}{\mathbb E[|\cK|]} \leq \frac{1}{10},
 \]
 where the last inequality holds for $r$ sufficiently large.

 Let $\mathcal E_2$ be the event that $|Q(\cK)| \geq \frac{1}{4} \mathbb E[|\cK|]$. By Markov's inequality
 \begin{equation} \label{mark1}
 \mathbb P [\mathcal E_2] \leq 4\frac{ \mathbb E [|Q(\cK)|]}{\mathbb E[|\cK|]}\leq 4\frac{ c_2p^4r^{17}}{pr^8}=\frac{1}{2}.
 \end{equation}

 Let $\mathcal E_3$ be the event that $\alpha_3(\cK) \geq r^{4+2\delta}$. An application of Lemma \ref{lem:whp} implies that this probability tends to zero as $r$ goes to infinity. In particular, by taking $r$ sufficiently large, we certainly have $\mathbb P[ \mathcal E_3] \leq \frac{1}{10}$.

 It therefore follows from the union bound that 
 \[
 \mathbb P[ \cE_1 \cup \cE_2 \cup \cE_3] \leq \frac{7}{10}.
 \]
 In particular, with positive probability, none of the three events occur. Finally, we note that $\neg ( \mathcal E_1 \cup \mathcal E_2 \cup \mathcal E_3 )$ implies that each of the three required properties at the beginning of the proof hold. 
    
\end{proof}

\section{A lower bound for Hadwiger-Debrunner numbers} \label{sec:HD}

For integers $p \geq q \geq 1$, a family of sets $\cF$ is said to satisfy the $(p,q)$-property if for any subset $\cF' \subset \cF$ with $|\cF'| \geq p$, there exist $q$ elements of $\cF'$ with a common intersection. In the context of this paper, the family of lines given by Theorem \ref{thm:BSlines} has the $(n^{\frac{4}{5}+\delta},3)$-property. The family $\cF$ is said to be \textit{pierced by $P$} if for all $F \in \cF$ there exists $x \in P$ such that $x \in F$. Solving a conjecture of Hadwiger and Debrunner \cite{HaDe}, it was proven by Alon and Kleitman \cite{AlKl} that for any $p \geq q \geq d+1$, there exists a constant $c_d(p,q)$ such that any family of convex sets in $\mathbb R^d$ with the $(p,q)$-property can be pierced by a set $P$ with size at most $c_d(p,q)$.

The problem of determining sharp bounds for these constants remains open. Keller and Smorodinsky \cite{KeSm} observed a connection with the Balogh-Solymosi Theorem and used the existence of the set of lines given by the dual formulation of Theorem \ref{thm:BS} to show that this constant must satisfy $c_2(p,3) =\Omega(p^{6/5-\epsilon})$ for all $\epsilon >0$. The improvement given in this paper by Theorem \ref{thm:BSlines} sharpens this bound.

\begin{corollary} \label{cor:HD} For all $\epsilon >0$ and all $p$ sufficiently large, there exists a family $\cF$ of convex sets in $\mathbb R^2$ with the $(p,3)$-property such that any set $P \subset \mathbb R^2$ piercing $\cF$ must have cardinality $\Omega(p^{\frac{5}{4}-\epsilon})$.
    
\end{corollary}

\begin{proof}
Fix $\epsilon>0$ and then choose $0<\delta<1/2$ to be sufficiently small with respect to $\epsilon$ (the simple choice $\delta=\epsilon$ is sufficient). The proof of Theorem \ref{thm:BSlines} gives, for all $r$ sufficiently large, a set $\cL_r$ of lines such that
\[
|\cL_r| \geq cr^5,\,\,\,\, \alpha_3(\cL_r) \leq r^{4+2\delta}.
\]
Let $p$ be a sufficiently large integer and set $r$ to be the largest integer such that $r^{4+2\delta} < p$. Then $\cL_r$ has the $(p,3)$-property. Note that, since $p \leq (r+1)^{4+2\delta} \leq (2r)^{4+2\delta} $, we have
\[
r > \frac{p^{\frac{1}{4+2\delta}}}{2}.
\]
On the other hand, since any point in $\mathbb R^2$ belongs to at most $3$ of the sets in $\cL_r$, it follows that any piercing set for the family $\cL_r$ must have cardinality at least
\[
\frac{|\cL_r|}{3} \gg r^5 \gg p^{\frac{5}{4+2\delta}} \geq p^{\frac{5}{4}- \epsilon}.
\]
\end{proof}

Using the common notation in this area, Corollary \ref{cor:HD} states that the bound $HD_2(p,3)= \Omega(p^{\frac{5}{4}-\epsilon})$ holds for all $p$ sufficiently large. Here, $HD_d(p,q)$ denotes the smallest value $c$ such that every family of convex sets in $\mathbb R^d$ with the $(p,q)$-property can be pierced by a set with size at most $c$, i.e. the optimal value of $c_d(p,q)$ which was shown to exist in \cite{AlKl}. The best known upper bound $HD_2(p,3)=O(p^{3+\epsilon})$ follows by combining the work of Keller, Smorodinsky and Tardos \cite{KST} with a bound of Rubin \cite{Ru} on the size of weak $\epsilon$-nets in the plane. 

In some of the literature concerning Hadwiger-Debrunner numbers, the definition of $HD_d(p,q)$ includes the additional restriction that the convex sets involved are also compact. In our case, we can transform the lines in $L$ into compact sets by simple truncating them at a sufficiently distant point so that none of the intersection properties are disturbed.


\section{Optimality of the Balogh-Solymosi supersaturation lemma} \label{sec:optimal}

In this section, we will prove that Theorem \ref{thm:ssl} is optimal up to logarithmic factors.


 \begin{proof}[Proof of Theorem \ref{thm:main}]

Consider first the set
\[
A:= \{ (x,y,x^2+y^2) :x,y \in \mathbb F_p \} \subset \mathbb F_p^3.
\]
This is a quadratic surface in $\mathbb F_p^3$, and any line containing at least three points of $A$ must be contained in $A$. It follows from the condition that $p$ is congruent to $3$ modulo $4$ that there are no lines contained in $A$. Indeed, suppose that a line
\[
\{(a_1,a_2,a_3)+t(v_1,v_2,v_3) : t \in \mathbb F_p \}
\]
is contained in $A$. We must have at least one of the $v_i$ not equal to zero, otherwise this is not a line. Considering $t \in \{-1,0,1\}$, it follows that
\begin{align*}
    a_3 & = a_1^2+a_2^2,
    \\ a_3 +v_3 &=(a_1+v_1)^2+(a_2+v_2)^2,
    \\ a_3-v_3 &= (a_1-v_1)^2+(a_2-v_2)^2.
\end{align*}
Expanding and simplifying yields
\begin{equation} \label{mod1}
v_1^2+v_2^2=v_3-2(a_1v_1+a_2v_2) = 2(a_1v_1+a_2v_2)-v_3,
\end{equation}
and hence we must have $v_1^2+v_2^2=0$. However, we cannot have $v_1=v_2=0$, as then \eqref{mod1} implies that $v_3=0$, which is not permitted. We may then assume without loss of generality that $v_2 \neq 0$, and the equation $v_1^2+v_2^2=0$ can then be rearranged into the form
\[
(v_1v_2^{-1})^2=-1,
\]
which cannot be valid with $p$ congruent to $3$ modulo $4$. A proof of this fact is also implicit in \cite[Lemma 5]{PoWo}. It thus follows that no line in $\mathbb F_p^3$ contains more than $2$ points from $A$. 

Now, let $\mathcal A$ denote the set $A$ viewed as a subset of $ \{0,1,\dots,p-1 \}^3 \subset \mathbb Z^3$. Define the grid set $G$ to be
\[
G:=  p\{0,1, \dots, r-1\}^3.
\]
The set $S$ from the statement of Theorem \ref{thm:main} is defined as
\[
S:= \mathcal A + G \subset \{0,\dots,pr-1 \}^3.
\]
After shifting by the vector $(1,1,1)$, this set lies in $[pr]^3$, and the number of collinear triples is unchanged. We can view $S$ as 
\[
S= \bigcup_{a \in \mathcal A} a + G,
\]
and note that this is a disjoint union, thus $|S|=p^2r^3$.

Let $\ell$ be a line $\mathbb R^3$ which contains at least three points from $S$. In particular, the line contains at least two points of $\mathbb Z^3$, and so we can express the line $\ell$ in the parameterised form
\[
\ell:= \{ x+tv : t \in \mathbb R \},
\]
with $v,x \in \mathbb Z^3$ and such that $v$ is the primitive direction of the line, by which we mean that $v=(v_1,v_2,v_3)$ is a nonzero vector and it satisfies the condition that the nonzero elements of $\{v_1,v_2,v_3\}$ have no common factors. Moreover, since we are only interested in integer points of this line, we may restrict our focus to the set $\{x+tv:t \in \mathbb Z \}$.

We first claim that $\ell$ intersects at most two of the sets $a+G$ with $a \in \cA$. Indeed, if $\ell$ intersects three distinct $a_i+G$ with $i=1,2,3$, then the line $\ell$ reduced to $\mathbb F_p^3$ (i.e. the set we get by taking the elements of $\ell \cap \mathbb Z^3$ and reducing each coordinate modulo $p$) contains each of $a_1, a_2,a_3 \in A$, which contradicts the fact that any line in $\mathbb F_p^3$ contains at most two elements of $A$.

It therefore follows that $\ell$ must contain at least two points from the same set $a+G$. Let $a+u_1$ and $a+u_2$ be two points of $S$ with $u_1,u_2 \in G$ and $u_1 \neq u_2$. Their difference is
\[
(a+u_1)-(a+u_2)=u_1-u_2=p(\alpha_1-\alpha_2),
\]
for some $\alpha_1,\alpha_2 \in \{0,1,\dots,r-1 \}^3$. It therefore follows that the primitive direction vector $v$ for the line $\ell$ satisfies $\|v\|_{\infty} \leq r$. We can thus henceforth restrict our focus to lines with small primitive direction.

Now, fix a direction vector $v$ with $\|v \|_{\infty}=q \leq r$. We will upper bound the number of collinear triples for lines with this direction, and then sum over all possible directions. We first need to count the number of lines with this direction which meet the grid $\{0,1,\dots, pr-1\}^3$. Such lines are determined by a ``first point" $x \in \ell \cap \{0,1,\dots, pr-1\}^3$, i.e. a point $x$ on the line and in the grid such that $x-v \notin \{0,1,\dots, pr-1\}^3$. Such a point $x$ must lie close to one of the faces of the cube containing $\{0,1,\dots, pr-1\}^3$, and it follows that there are $O(q(pr)^2)$ choices for $x$, and thus $O(q(pr)^2)$ lines with direction $v$ to consider.

Now, fix one such line $\ell$ with primitive direction $v$ intersecting $\{0,1,\dots, pr-1\}^3$. Since we are only interested in points of this line belonging to the grid $\{0,1,\dots, pr-1\}^3$, we may restrict our attention to the subset
\begin{equation} \label{reparamaterise}
    \{x+tv: t \in \mathbb Z \cap [0,pr/q]  \}
\end{equation}
of the line. How many elements from the set in \eqref{reparamaterise} belong to $S$? Recall that $\ell$ intersects at most $2$ sets of the form $a+G$ with $a \in \mathcal A$, and let us fix these sets as $a_1+G$ and $a_2+G$. Consider the set \eqref{reparamaterise} reduced modulo $p$. This is a line in $\mathbb F_p^3$, with points possibly occurring with multiplicity as we cycle repeatedly over the line. We can only hit a point of $S$ when the corresponding line in $\mathbb F_p^3$ hits either $a_1$ or $a_2$, which happens a proportion of $2/p$ of the time. It therefore follows that the line $\ell$ hits at most $O\left ( \frac{1}{p} \cdot \frac{pr}{q} \right )=O(r/q)$ of the points of $S$. The direction $v$ therefore contributes at most
\[
O(q(pr)^2 \cdot (r/q)^3) = O(p^2r^5/q^2)
\]
collinear triples to the count. Finally, we dyadically decompose over all possible sizes of $q$ and sum this quantity to conclude that
\begin{align*}
    t(S) \ll \sum_{j \in \mathbb N : 2^{j-1} \leq r} \sum_{v \in \mathbb Z^3 : 2^{j-1} \leq \|v\|_{\infty}\leq 2^{j}}  \frac{p^2r^5}{(2^j)^2} \ll \sum_{j \in \mathbb N : 2^{j-1} \leq r}(2^j)p^2r^5 \ll p^2r^6.
\end{align*}

\end{proof}

We now verify that this construction implies the claim of Corollary \ref{cor:main}.

\begin{proof}[Proof of Corollary \ref{cor:main}]
    We first check the boundary cases $s=1$ and $s=0$. For $s=1$, we simply take $S$ to be the set $\mathcal A$ defined in the proof of Theorem \ref{thm:main}, and set $n=p$. The fact that this set contains no collinear triples was proven above, and the same construction appeared in work of P\'{o}r and Wood \cite{PoWo}. For $s=0$, we take $S=[n]^3$. It is a well-known fact that $[n]^3$ contains $\Theta(n^6)$ collinear triples; see for instance \cite[Claim 4.1]{BaSo}.

    Now let $0 < s < 1$, let $p$ be a sufficiently large prime congruent to $3$ modulo $4$ and set
    \[
    r:=  \left \lceil p^{\frac{1-s}{s}} \right \rceil.
    \]
    Apply Theorem \ref{thm:main} with this choice of $p$ and $r$, and so $n=pr$. Note that
    \[
    p^{1/s} \leq n \leq 2 p^{1/s},
    \]
    where the upper bound follows from the assumption that $p$ is sufficiently large (with respect to $s$). A rearrangement gives
    \[
    \frac{n^{s}}{2^s} \leq p \leq n^s.
    \]
    It follows from Theorem \ref{thm:main} that
    \[
    |S|=p^2r^3=n^3/p \geq n^{3-s}
    \]
    and
    \[
    t(S) \ll p^2r^6= n^6/p^4 \leq 2^{4s}n^{6-4s} \ll n^{6-4s}.
    \]
\end{proof}

\section*{Acknowledgments and AI Disclaimer} Oliver Roche-Newton is partially supported by the Austrian Science Fund (FWF) project PAT2559123. Thank you to J\'{o}zsef Balogh, Chaya Keller, Zuzana Pat\'akov\'a, Cosmin Pohoata, Shakhar Smorodinsky, J\'{o}zsef Solymosi and Audie Warren for helpful discussions.

This work was carried out in collaboration with ChatGPT. The paper is human-written and the author takes full responsibility for the contents of the paper.

\bibliography{bibliography}
\bibliographystyle{plain}
\end{document}